\documentclass[10pt,letterpaper,twocolumn]{article}

\usepackage{authblk}
\usepackage{abstract}
\usepackage{titlesec}
\usepackage{geometry}
\usepackage{multicol}
\usepackage{graphicx}
\usepackage{graphics,epstopdf}
\usepackage{graphicx}
\usepackage{float}
\usepackage{epstopdf}
\usepackage[labelsep=period,font=it]{caption}
\usepackage{booktabs}
\usepackage[hang, flushmargin]{footmisc}
\usepackage[fleqn]{amsmath}
\usepackage{mathtools}
\usepackage{natbib}
\usepackage{graphics} 
\usepackage{mathptmx} 
\usepackage{amsmath} 
\usepackage{amssymb}  
\usepackage{float}
\usepackage[english]{babel}
\usepackage{enumitem}
\newtheorem{theorem}{Theorem}
\newtheorem{assumption}{Assumption}

 \newenvironment{proof}{\emph{Proof:}}{\hspace{\stretch{1}}\rule{1ex}{1ex}}

\newtheorem{remark}{Remark}

\graphicspath{{./figures/}}

\titleformat*{\section}{\normalsize\bf}
\titlespacing{\section}{0pt}{18pt}{6pt}
\titleformat*{\subsection}{\normalsize\it}
\titlespacing{\subsection}{0pt}{12pt}{6pt}

\providecommand{\keywords}[1]{\hspace{0.25in}\textit{Keywords}\\
\vspace{-20pt}
\begin{center}
\begin{minipage}{5.84in}
#1
\end{minipage}
\end{center}}

\title{\fontsize{14}{16.8}
\textbf{Optimal control of a linearized continuum model for re-entrant manufacturing production systems}
}
\author[1]{Xiaodong Xu}
\author[1]{Stevan Dubljevic \thanks{To whom all correspondence should be addressed \textit{Stevan.Dubljevic@ualberta.ca}}}
\affil[1]{Department of Chemical and Materials Engineering, University of Alberta,
Edmonton, Alberta, Canada}
\date{}
\begin{document}
\pagenumbering{gobble}
\twocolumn[
\begin{@twocolumnfalse}
\maketitle
\begin{abstract}
\vspace*{-10pt}
\noindent
A re-entrant manufacturing system producing a large number of items and involving many steps can be approximately modeled by a hyperbolic partial differential equation (PDE) according to mass conservation law with respect to a continuous density of items on a production process. The mathematic model is a typical nonlinear and nonlocal PDE and the cycle time depends nonlinearly on the work in progress. However, the nonlinearity brings mathematic and engineering difficulties in practical application. In this work, we address the optimal control based on the linearized system model and in order to improve the model and control accuracy, a modified system model taking into account the re-entrant degree of the product is utilized to reflect characteristics of small-scale and large-scale multiple re-entrant manufacturing systems. In this work, we solve the optimal output reference tracking problem through combination of variation approach and state feedback internal model control (IMC) method. Numerical example on optimal boundary influx for step-like demand rate is presented. In particular, the demand rates are generated by an known exosystem.
\end{abstract}
\vspace{12pt}
\end{@twocolumnfalse}
]
\saythanks
\section*{Introduction}
\noindent The re-entrant manufacturing system is one of the most
complex manufacturing processes, and it has the following characteristics:
there are a large number of loadings, large quantities of machines and production steps,
and a high degree of re-entry in the system. Several model methods have been provided in literatures for multiple re-entrant production flows: 
Petri-net, queuing network, fluid network, and partial differential equations (PDE). In this work, the model provided in \cite{armbruster2006continuum} is utilized and the expression of the velocity in the model is improved by \cite{dong2011optimal}. \\ 
\indent The output regulation problem or servo-problem is one classical and essential control problem. The problem is formulated as regulator design for the fixed plant such that the controlled output tracks a desired reference signal (and/or reject disturbance) generated by an exosystem. In order to generalize the well-developed theory of finite-dimensional systems to infinite-dimensional systems, significant efforts have been made: the geometric methods developed in \cite{francis1977linear} in finite-dimensional systems. Recently, these geometric methods were introduced for non-spectral hyperbolic systems in \cite{xu2015output} and \cite{Xu2016servo}. Moreover, along the same line, finite-dimensional output and error feedback regulators solving the regulation problem are introduced in \cite{Xu2015finites} and \cite{xu2016finite}. In particular, the boundary tracking actuation approach was developed in \cite{deutscher2015backstepping}. In this work, we are deriving finite time optimal control results for boundary controlled production system by using weak variations in \cite{athans1966optimal} and constructing a boundary optimal regulator to address the demand rate tracking problem. \\
\indent The remainder of this work is organized as follows: some notations used throughout this work are introduced in the next section. In the third section, the problem is stated and then the main results are demonstrated in the fourth and fifth section, where the boundary optimal controller and the boundary tracking regulator are designed. The proposed approach is verified through numerical examples in the sixth section.\vspace{-10mm}\\ 
\section*{Mathematical Preliminaries and Notation}
\noindent In order to demonstrate the work more clearly, we first introduce some preliminary mathematics. The following linear operator used in this work is defined by
\[P(f(z)): = \int_0^1 {P(z,y)f(y)} dy,z \in \left[ {0,1} \right]\]
The inner product in this work is denoted by
\[\left\langle {f(z),g(z)} \right\rangle  = \int_0^1 {f(z)g(z)} dz\]
Partial derivatives with respect to time and spatial variables $t$ and $z$ are presented by
${\partial _t} = \frac{\partial }{{\partial t}},{\partial _z} = \frac{\partial }{{\partial z}}$.\vspace{-5mm}
\section*{Problem Formulation}
Assuming the mass conservation law and scaling the spatial variable $z\in [0,1]$, the large-scale re-entrant manufacturing systems can be modelled by the continuity equation:
\begin{equation}\label{pl-1}
\partial_t \rho(z,t) + \partial_z \left( {v\left( {\rho (z,t)} \right)\rho (z,t)} \right) = 0
\end{equation}
where $\rho(z,t)$ describes the density of products at stage $z$ of production at a time $t$ and $v(\rho(z,t))$ is a velocity function depending on the density $\rho(z,t)$ only. The product rates $u(t)$ and $y(t)$ entering and existing the production system at $z=0$ and $z=1$ are defined as follows:
\begin{equation}\label{pl-2}
u(t) = {\left. {v\left( {\rho (z,t)} \right)\rho (z,t)} \right|_{z = 0}}
\end{equation}
\begin{equation}\label{pl-3}
y(t) = {\left. {v\left( {\rho (z,t)} \right)\rho (z,t)} \right|_{z = 1}}
\end{equation}
The total load in the production line $L(t)$ (work-in-progress (WIP)) is defined by $L(t) = \int_0^1 {\rho (z,t)dz}$ and the velocity $v(\rho(z,t))$ is chosen as the function of $L(t)$ to describe the equilibrium velocity of the factory. It should be noted that the velocity function is bounded, positive and monotonically decreasing. Obviously, we have the nonnegative influx $u(t)$ and nonnegative initial data $\rho(z,0)=\rho_0(z)$, the density will definitely remain nonnegative. Usually, the velocity can be described by
\begin{equation}\label{v-1}
v(\rho (z,t)) = {v_0}\left( {1 - \frac{{L(t)}}{{{L_{\max }}}}} \right)
\end{equation}
where $v_0$ denotes the empty system velocity and $L_{\max}$ is the maximal load of the factory. Naturally, the velocity is determined by the total WIP. The production process can be described as an equivalent M/M/1 queue. Let $\rho_{ss}$ denote the steady state density and then the mean cycle time in steady state is $\tau_{ss}=1/v_{ss}$. Since the steady state velocity is the same for all items in the queue, the time that an item spends in the machines without waiting is $1/v_{\max}$. Based on queuing theory, the cycle time $\tau=(1+L)/v_{\max}$ is obtained, the steady state velocity thus becomes:
\begin{equation}\label{v-2}
v(\rho (z,t)) = \frac{{{v_{\max }}}}{{1 + \int_0^1 {\rho (z,t)dz} }} = \frac{{{v_{\max }}}}{{1 + L(t)}}
\end{equation}
The expression in (\ref{v-2}) is widely utilized to describe the relation between $v$ and $\rho$ for large-scale multiple re-entrant systems. Obviously, due to the integration in $L(t)$, the velocity $v$ is only function of time $t$, equation (\ref{pl-1}) can be rewritten as:
\begin{equation}\label{newpl-1}
{\partial _t}\rho (z,t) + v{\partial _z}\left( {\rho (z,t)} \right) = 0,z \in [0,1],t > 0
\end{equation}
In reality, the production system velocity depends not only on the WIP, but also on the re-entrant factor $\alpha$ defined as the ratio of the product processing time of re-entrant steps and the total processing time, i.e. $\alpha=\frac{P_1}{P_{1}+P_{2}}=\frac{P_1}{P_{tol}}$. Here, $P_1$ is the re-entrant processing time and $P_2$ is the non-re-entrant processing time. $P_{tol}$ denotes the total processing time. In the non-re-entrant process, let $m$ denote the number of total workstations, then the improved velocity is given by:
\begin{equation}\label{newv-2}
v = \frac{{{v_{\max }}}}{{1 + \left( {{\alpha ^2} + {{\left( {1 - \alpha } \right)}^2}/m} \right)L(t)}}=\Phi (L(t))
\end{equation}
In the following, we start to do the linearization of the model (\ref{pl-2}), (\ref{pl-3}) and (\ref{newpl-1}). To this end, let
\[\begin{array}{l}
\tilde \rho (z,t): = \rho (z,t) - \bar \rho ,\tilde L(t): = L(t) - \bar \rho ,{{\tilde \rho }_0}(z): = {\rho _0}(z) - \bar \rho ,\\
\tilde v(t): = v\left( {\bar \rho  + \tilde L(t)} \right),\tilde u(t): = \tilde v(t)\tilde \rho (0,t),\tilde y(t): = \tilde v(t)\tilde \rho (1,t).
\end{array}\]
Then, the system can be rewritten as follows:
\[\begin{array}{l}
{\partial _t}\tilde \rho (z,t) + \tilde v(t){\partial _z}\tilde \rho (z,t) = 0,z \in (0,1),t > 0\\
\tilde \rho (0,t) = {{\tilde \rho }_0}(z),z \in (0,1),\\
\tilde u(t) = \tilde v(t)\tilde \rho (0,t),\\
\tilde y(t) = \tilde v(t)\tilde \rho (1,t).
\end{array}\]
Taking Taylor expansion of $\tilde v(t)$ at $L(t)=\bar \rho$ gives: $v\left( {\rho (z,t)} \right) = \Phi (L(t)) = \Phi (\bar \rho ) + \Phi '(\bar \rho )\left( {L(t) - \bar \rho } \right) + {\rm O}\left( {{{(L(t) - \bar \rho )}^2}} \right)$.
If we use the first term of taylor expansion of $v(\rho)$ to approximate itself, the linearized model around $\bar \rho$ is then given by:
\begin{equation}\label{L-pl}
\begin{array}{l}
{\partial _t}\tilde \rho (z,t) + v(\bar \rho ){\partial _z}\tilde \rho (z,t) = 0,z \in (0,1),t > 0\\
\tilde \rho (0,t) = {{\tilde \rho }_0}(z),z \in (0,1),\\
\tilde u(t) = v(\bar \rho )\tilde \rho (0,t),\\
\tilde y(t) = v(\bar \rho )\tilde \rho (1,t).
\end{array}
\end{equation}
In order to avoid huge error caused by the linearization, we add the following model improvement strategy:\vspace{-4mm}
\begin{figure}[H]
  \centering
  \includegraphics[width=2.5in]{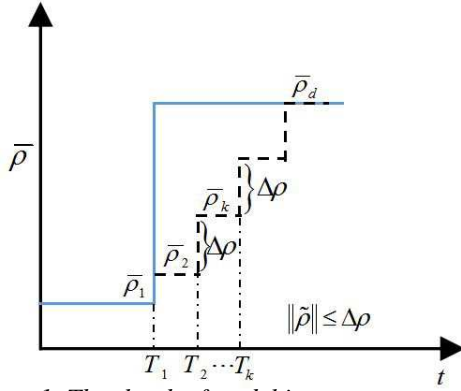}\vspace{-5mm}\\
  \caption{The sketch of model improvement strategy.}\label{fig-2}\vspace{-5mm}
\end{figure}
As shown in Figure \ref{fig-2}, when the average density reaches to $\bar \rho_d$ from $\bar \rho_1$, we can divide the difference between $\bar \rho_d$ and $\bar \rho_1$ into number $d$ intervals. Then, we can chose $v(\bar \rho_1)$ as the velocity from $T_1$ to $T_2$ and chose $v(\bar \rho_k)$ from $T_{k-1}$ to $T_k$, i.e. after the density reaches to $\bar \rho_{k-1}$. Therefore, we rewrite the model (\ref{L-pl}) as:
\begin{equation}\label{new-PL}
\begin{array}{*{20}{l}}
{\partial _t}\tilde \rho (z,t) + v({{\bar \rho }_k}){\partial _z}\tilde \rho (z,t) = 0,z \in (0,1),t > 0\\
\tilde \rho (0,t) = {{\tilde \rho }_0}(z),z \in (0,1),\\
\tilde u(t) = v({{\bar \rho }_k})\tilde \rho (0,t),\\
\tilde y(t) = v({{\bar \rho }_k})\tilde \rho (1,t).
\end{array}
\end{equation}
with $k=1,\cdots,d$ and $d = \left( {{{\bar \rho }_d} - {{\bar \rho }_1}} \right)/\Delta \rho  + 1$.
The final and vital target in manufacturing is the controlling the production rate of production systems. If we produce too much of an item, holding cost/stocktaking occurs and while producing to little of an item will result in lost in sales. In order to increase profitability, it is important for a production system to match its projected demand optimally. Although demand may be stochastic, there are numerous ways to generates the demand forecast for the next day, week, month, etc. Accordingly, the objectives in this work are as follows:
 \begin{enumerate}[leftmargin=1cm]
\item [\textbf{(i)}.] Since the system (\ref{new-PL}) is a boundary controlled hyperbolic PDE system, the first important mission is to guarantee the closed-loop system stability. In this work, a boundary optimal state feedback regulator is designed to achieve the optimal stability.
\item [\textbf{(ii)}.] In order to match the projected demand $d_r$, the boundary tracking regulator has to be designed. In other words, we have to realize:
    \begin{equation}\label{target}
    \mathop {\lim }\limits_{t \to  + \infty } e(t) = \mathop {\lim }\limits_{t \to  + \infty } \left( {y(t) - {d_r}} \right) = 0
    \end{equation}
   Without loss of generality, in Figure \ref{fig-2}, we can use $\bar \rho_d$ to express $d_r$ as: $dr= v(\bar \rho_d)\bar \rho_d$.
   \indent To achieve (\ref{target}), we need to design boundary controllers at every $k$th step, for $k=1,\cdots, d$. Since we divide the difference $\bar \rho_d-\bar \rho_1$ into same intervals, i.e. $\Delta \rho$, we just need to control the linearized system (\ref{new-PL}) at every $k$th step such that the density $\tilde \rho$ achieves to $\Delta \rho$. Finally, if we add all resulting control signals from every $k$th step to the original nonlinear model (\ref{pl-1})--(\ref{pl-3}), the density $\rho$ in (\ref{pl-1})--(\ref{pl-3}) can reach $\bar \rho_d$. In this work, we employ the idea from IMC to achieve the boundary tracking control, i.e. $\mathop {\lim }\limits_{t \to  + \infty } {\tilde e_\rho }(t) = \mathop {\lim }\limits_{t \to  + \infty } \left( {\tilde \rho (1,t) - \Delta \rho } \right) = 0$ and thus $\mathop {\lim }\limits_{t \to  + \infty } \tilde e(t) = \mathop {\lim }\limits_{t \to  + \infty } \left( {\tilde y(t) - v({{\bar \rho }_k})\Delta \rho } \right) = 0$.
\end{enumerate}\vspace{-5mm}

\section*{Optimal Controller Design}
For the linearized model (\ref{new-PL}), we first define a cost functional $J$ by
\begin{equation}\label{cost-1}
\begin{array}{l}
J(\tilde \rho ,\tilde u): = \frac{1}{2}\int_0^T {\left[ {\left\langle {\tilde \rho (z,t),{q_1}(\tilde \rho (z,t))} \right\rangle  + R{{\tilde u}^2}(t)} \right]dt} \\
\hspace{15mm} + \frac{1}{2}\left\langle {\tilde \rho (z,T),{P_f}\left( {\tilde \rho (z,T)} \right)} \right\rangle
\end{array}
\end{equation}
Here, the symbols $q_1\ge 0$, $R> 0$, and $P_{f}\ge 0$ are weighting kernels for states, input and terminal states of the closed-loop system. In particular, the positivity of $R$ is used to guarantee the boundedness of control signals.

\subsection*{Open-loop Controller}
We minimize the cost functional $J$ subject to the constraints introduced by PDE-dynamics, i.e. the following optimization:
\begin{equation}\label{min-1}
\min J(\tilde \rho ,\tilde u){\text{ subject to}}
\end{equation}
\begin{equation}\label{min-2}
{\partial _t}\tilde \rho (z,t) + v({{\bar \rho }_k}){\partial _z}\tilde \rho (z,t) = 0
\end{equation}
\begin{equation}\label{min-3}
\tilde u(t) = v({{\bar \rho }_k})\tilde \rho (0,t)
\end{equation}
\begin{equation}\label{min-4}
\tilde \rho (z,0) = {{\tilde \rho }_0}(z)
\end{equation}
\begin{equation}\label{min-5}
v({{\bar \rho }_k}) = \frac{{{v_{\max }}}}{{1 + \left( {{\alpha ^2} + {{(1 - \alpha )}^2}/m} \right){{\bar \rho }_k}}}
\end{equation}
The following theorem provides necessary conditions such that the above constraint minimization problem can be solved and the open-loop control problem of (\ref{new-PL}) in finite-time horizon can be addressed.
\begin{theorem}\label{theorem-1}
Consider the linear hyperbolic PDEs given by (\ref{new-PL}) defined on the finite-time horizon $t\in [0,T]$ and the cost function (\ref{cost-1}). If we define the nominal states, control and co-states that minimize the cost function as: $\tilde \rho^*(z,t)$, $\tilde u^*(t)$, and $\lambda(t)$, then the necessary conditions for optimality are as follows:
\begin{equation}\label{theo1-1}
{\partial _t}{{\tilde \rho }^*}(z,t) + v({{\bar \rho }_k}){\partial _z}{{\tilde \rho }^*}(z,t) = 0
\end{equation}
\begin{equation}\label{theo1-2}
{\partial _t}\lambda  + v({{\bar \rho }_k}){\partial _z}\lambda  + {q_1}({{\tilde \rho }^*}) = 0
\end{equation}
with boundary conditions:
\begin{equation}\label{theo1-B1}
{{\tilde u}^*}(t) = v({{\bar \rho }_k}){{\tilde \rho }^*}(0,t),\lambda (1,t) = 0
\end{equation}
and initial/terminal conditions:
\begin{equation}\label{theo1-I1}
{{\tilde \rho }^*}(z,0) = {{\tilde \rho }_0}(z),\lambda (z,T) = {P_f}\left( {{{\tilde \rho }^*}(z,T)} \right)
\end{equation}
where the optimal control input is:
\begin{equation}\label{theo1-OI}
{{\tilde u}^*} =  - \frac{1}{R}\lambda (0,t)
\end{equation}
\end{theorem}
Proof: omitted due to page limit.

\subsection*{State-feedback Controller}
Now, we are considering the state-feedback controller design problem. First, we define the following linear transformation that relates the co-state $\lambda$ to the state $\tilde \rho$:
\begin{equation}\label{Co-stat}
\lambda (z,t) = \int_0^1 {P(z,y,t)} {{\tilde \rho }^*}(y,t)dy
\end{equation}
Moreover, the terms in previous section are denoted by:
\[\begin{array}{l}
{q_1}\left( {\tilde \rho^*(z,t)} \right) = \int_0^1 {{q_1}(z,y)\tilde \rho^*(y,t)dy} \\
{P_{{f}}}\left( {\tilde \rho^*(z,T)} \right) = \int_0^1 {{P_{{f}}}(z,y)\tilde \rho^*(y,T)dy}
\end{array}\]
Then, we have the following result for the boundary controlled linear coupled hyperbolic PDE systems.
\begin{theorem}\label{theorem-2}
The optimal boundary control in state-feedback form is given by:
\begin{equation}\label{theo2-1}
{{\tilde u}^*} =  - \frac{1}{R}\int_0^1 {P(0,y,t){{\tilde \rho }^*}(y,t)dy}
\end{equation}
where the spatial varying transformation kernel $P(z,y)$ is the solution of the following differential Riccati equations:
\begin{equation}\label{theo2-r1}
\begin{array}{l}
{\partial _t}P(z,y,t) +  v({{\bar \rho }_k}){\partial _z}P(z,y,t) + v({{\bar \rho }_k}){\partial _y}P(z,y,t)\\
 + {q_1}(z,y) - \frac{1}{R}P(z,0,t)P(0,y,t) = 0
\end{array}
\end{equation}
with boundary conditions:
\begin{equation}\label{theo2-rb1}
P(z,1,t) = 0,P(1,y,t) = 0
\end{equation}
and terminal condition:
\begin{equation}\label{theo2-ri2}
P(z,y,T) = P_f(z,y)
\end{equation}
\end{theorem}
\begin{proof}
To proof this theorem, one just needs to evaluate $\lambda$ in (\ref{theo1-2}), (\ref{theo1-B1}) and (\ref{theo1-I1}) using the linear transformation (\ref{Co-stat}). One boundary condition for $P(1,y,t)$ is directly resulted from the boundary conditions in (\ref{theo1-B1}) and the other boundary condition for $P(z,1,t)$ arises from integration by parts.
\end{proof}
Consequently, for the infinite-time horizon, the steady-state solution of (\ref{theo2-r1})--(\ref{theo2-ri2}) is given by:
\[\begin{array}{l}
 v({{\bar \rho }_k}){\partial _z}{P_{ss}}(z,y) + v({{\bar \rho }_k}){\partial _y}{P_{ss}}(z,y)\\
 + {q_1}(z,y) - \frac{1}{R}{P_{ss}}(z,0){P_{ss}}(0,y) = 0\\
{P_{ss}}(z,1) = 0,{P_{ss}}(1,y) = 0
\end{array}\]
and the time-invariant steady-feedback control law is given by:
$${{\tilde u}^*} =  - \frac{1}{R}\int_0^1 {P_{ss}(0,y){{\tilde \rho }^*}(y,t)dy}$$

\section*{Demand Tracking}
In this section, we will apply the technique motivated by IMC to reduce the mismatch between the outflux and a demand rate target. Once the demand rates are predicted, we are able to find out the value of $v(\bar \rho_k)\Delta \rho$ according to different values of number $d$ given in Figure \ref{fig-2}. Therefore, we can use an finite-dimensional exosystem to generate the signal $v(\bar \rho_k)\Delta \rho$ for the controller of the production system. The exosystem can be defined by:
\begin{equation}\label{exo-1}
\dot w(t) = Sw(t),w(0) \in {\mathbb C^n}
\end{equation}
\begin{equation}\label{exo-2}
v(\bar \rho_k)\Delta \rho = q_rw(t),t\ge 0
\end{equation}
with $q_r$ matrix of appropriate dimensions which is assumed to be known for the regulator design.
\begin{assumption}\label{asump-1}
$S:D(S)\subset \mathbb{C}^n\to \mathbb{C}^n$ is a diagonal or diagonalizable matrix having all its eigenvalues on the imaginary axis, i.e. $\sigma(S)=(\lambda_k)_{k=1,...,n}$. In particular, one of $\lambda_k$ can be chosen as zero. Moreover, we assume that $(\phi_{k})_{k=1,...,n}$ are eigenvectors of $S$ and form an orthonormal basis of $\mathbb{C}^n$. This allows the modeling of steplike and sinusoidal exogenous signals.
\end{assumption}
To solve the demand rate tracking problem, the optimal state feedback controller with a feedforward of the signal model states is considered:
\begin{equation}\label{vt-cotr}
\tilde u(t) =  - \frac{1}{R}\int_0^1 {P_{ss}(0,y)\tilde \rho (y,t)dy}  + v({{\bar \rho }_k})m_w^Tw(t)
\end{equation}
The feedback gain $P(0,y)$ is a solution of differential equations in Theorem \ref{theorem-2} and the feedforward gain $m_v^T$ has to be determined. We have the following result which provides a choice of $m_v^T$.
\begin{theorem}\label{theorem-3}
The feedforward gain for the signal model states has the following form:
\begin{equation}\label{fw-gain}
m_w^T = \frac{1}{{v({{\bar \rho }_k})R}}\int_0^1 {P_{ss}(0,y){m^T}(y)dy + } {m^T}(0)
\end{equation}
such that the tracking contorl can be achieved, where the spatial varying vector $m^T(z)$ is the solution of the following regulator equations:
\begin{equation}\label{theo3-1}
v({{\bar \rho }_k}){d_z}{m^T}(z) + {m^T}(z)S = 0
\end{equation}
with boundary conditions:
\begin{equation}\label{theo3-3}
{m^T}(1) = \frac{1}{{v(\bar \rho_k )}}q_r^T
\end{equation}
\end{theorem}
\begin{proof}
In order to determine the feedforward gain $m_v^T$, we introduce for (\ref{new-PL}) and (\ref{exo-1}) error states:
\begin{equation}\label{err-sta}
e(z,t) = \tilde \rho (z,t) - {m^T}(z)w(t)
\end{equation}
where $m^T(z)$ has to be found. By applying (\ref{new-PL}), (\ref{exo-1}) and (\ref{err-sta}), one obtains:
\begin{equation}\label{st-er-1}
{\partial _t}e(z,t) + v({{\bar \rho }_k}){\partial _z}e(z,t) = 0
\end{equation}
if $m^T(z)$ satisfies the following condition:
\begin{equation}\label{m-cd1}
v({{\bar \rho }_k}){d_z}{m^T}(z) + {m^T}(z)S = 0
\end{equation}
The boundary condition for $e(z,t)$ is given by:
\begin{equation}\label{st-er-B}
v({{\bar \rho }_k})e(0,t) =  - \frac{1}{R}\int_0^1 {P_{ss}(0,y)e(y,t)dy}
\end{equation}
if $m^T(0)$ satisfies the following condition:
\[m_w^T = \frac{1}{{v({{\bar \rho }_k})R}}\int_0^1 {P_{ss}(0,y){m^T}(y)dy + } {m^T}(0)\]
Finally, the tracking error $\tilde e(t)$ becomes:
\begin{equation}\label{track-error}
\tilde e(t) = v({{\bar \rho }_k})e(1,t)
\end{equation}
if the following condition holds:
\[{m^T}(1) = \frac{1}{{v({{\bar \rho }_k})}}q_r^T\]
According to the illustration in previous sections, the error system (\ref{st-er-1}) with (\ref{st-er-B}) is stable optimally and therefore the tracking error $e(t)$ in (\ref{track-error}) decays to zero optimally, which proves the tracking control is achieved. Moreover, summarizing equations with respect to $m^T(z)$ yields the conclusion of the theorem. This concludes the proof.
\end{proof}
\begin{remark}\label{remark-1}
Rewriting the control law in (\ref{vt-cotr}) gives:
\begin{equation}\label{new-cot}
\begin{array}{l}
\tilde u(t) =  - \frac{1}{R}\int_0^1 {P_{ss}(0,y)\tilde \rho (y,t)dy} \\
\hspace{10mm} + \left( {\frac{1}{R}\int_0^1 {P(0,y){m^T}(y)dy + } v({{\bar \rho }_k}){m^T}(0)} \right)w(t)
\end{array}
\end{equation}
The exosystem is constructed according to the predicted demand rate, the state of exosystem $w(t)$ is easily obtained and is known to the controller (\ref{new-cot}). However, the full state $\tilde \rho(z,t)$ is not easily obtained. In this case, we can apply the control law:
\begin{equation}\label{new-cot-2}
\begin{array}{l}
\tilde u(t) =  - \frac{1}{R}\lambda(0,t) \\
\hspace{10mm} + \left( {\frac{1}{R}\int_0^1 {P_{ss}(0,y){m^T}(y)dy + } v({{\bar \rho }_k}){m^T}(0)} \right)w(t)
\end{array}
\end{equation}

\end{remark}\vspace{-9mm}
\section*{Numerical Simulation}
In this section, we will use a mini-fab model shown in Figure \ref{fig-3} to study the proposed approach. The mini-fab model is chosen since it has all important characteristics of the re-entrant manufacturing systems such as batch production, re-entrant and in particular different processing time.
\begin{figure}[ht]
  \centering
  \includegraphics[width=3in]{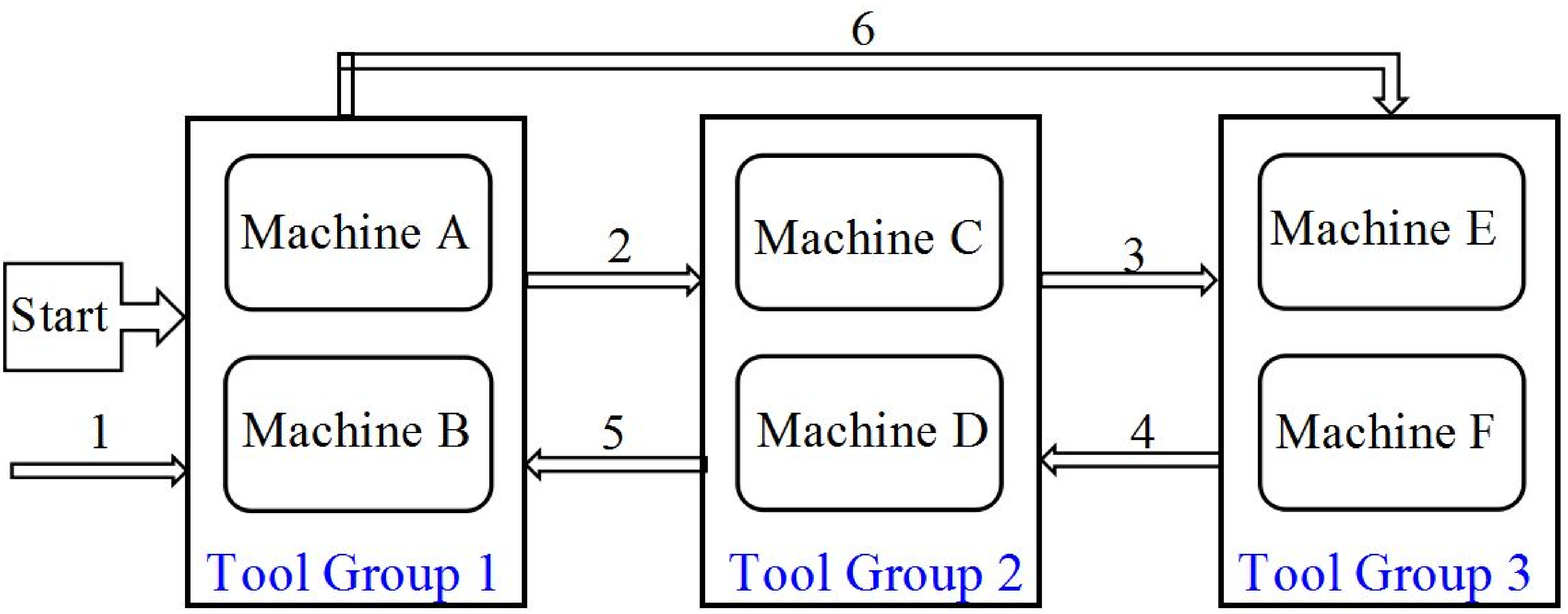}\vspace{-4mm}\\
  \caption{Process flow diagram of the mini-fab.}\label{fig-3}
\end{figure}
Now, we assume that there is a product in the mini-fab model and in Table \ref{table-1} processing steps and processing time are given.
\begin{table}[ht]
\caption{Processing time of product at each step.\vspace{-4mm}} 
\centering 
\begin{tabular}{c c c c c c} 
\hline\hline 
\hspace{-3mm}Machining centres&Processing time (hours) \\ 
\hline
\hspace{-6mm}Machine A $\&$ B&Step 1: 1.5&Step 5: 1.5\\ %
\hspace{-6mm}Machine C $\&$ D&Step 2: 0.5&\hspace{-3mm}Step 4: 1\\
\hspace{-7mm}Machine E $\&$ F&\hspace{-3mm}Step 3: 1&Step 6: 0.5\\
\hline 
\end{tabular}
\label{table-1} 
\vspace{-2mm}
\end{table}
According to Figure \ref{fig-3}, $m=3$ and the steps 4-6 are re-entrant steps. Table \ref{table-1}
provides: $P_1=3$ (hours) and $P=6$ (hours), so the re-entrant factor is $\alpha=0.5$ and $v_{\max}=1/P=4$ (units/day). Without loss of generality, we assume that the system stays in steady-state with influx given in (\ref{pl-2}): $u(t)=4$ (units/day). Therefore, the corresponding steady density $\bar \rho_1$ is the solution of the equation: $v({{\bar \rho }_1}){{\bar \rho }_1} = \frac{{{v_{\max }}}}{{1 + ({\alpha ^2} + {{(1 - \alpha )}^2}/m){{\bar \rho }_1}}}{{\bar \rho }_1} = 4$ and therefor ${{\bar \rho }_1} = 1.5$. We assume the demand rate to be tracking is $d_r(t)=\frac{60}{11}$ and thus $\rho_d=2.5$. As shown in Figure \ref{fig-2}, we discuss the performances of the proposed approach by assuming the different values of $d$, e.g. $d=2$ or $d=3$, see Table \ref{table-2}.\vspace{-5mm}
\begin{table}[ht]
\caption{Different values of $d$.\vspace{-4mm}} 
\centering 
\begin{tabular}{c c c c c c} 
\hline\hline 
$d$&$\Delta \rho$&$v(\bar \rho_k)$ \\ 
\hline
$d=2$&$\Delta \rho=1$&$v(\bar \rho_1)=8/3$\\ %
$d=3$&$\Delta \rho=0.5$&$v(\bar \rho_1)=8/3$,$v(\bar \rho_2)=2.4$\\
\hline 
\end{tabular}
\label{table-2} 
\vspace{-2mm}
\end{table}
Based on the conclusion of previous theory sections, we start applying boundary optimal tracking control law to the original nonlinear model (\ref{pl-1})--(\ref{pl-3}) motivated by (\ref{vt-cotr}). According to different division number $d$, we have the following discussion:\\
For the case: $d=3$, we need to solve the following differential Riccati equations and regulator equation:
\[\begin{array}{*{20}{l}}
v({{\bar \rho }_k}){\partial _z}{P_{kss}}(z,y) + v({{\bar \rho }_k}){\partial _y}{P_{kss}}(z,y)\\
 + {q_1}(z,y) - \frac{1}{R}{P_{kss}}(z,0){P_{kss}}(0,y) = 0\\
v({{\bar \rho }_k}){d_z}m_k^T(z) + m_k^T(z)S = 0\\
{P_{kss}}(z,1) = 0,{P_{kss}}(1,y) = 0\\
m_k^T(1) = \frac{1}{{v({{\bar \rho }_k})}}q_r^T,k = 1,2
\end{array}\]
and then the control law is given by:
\begin{equation}\label{cou-1}
u_1(t) = \left\{ \begin{array}{l}
v({{\bar \rho }_1}){{\bar \rho }_1} - \frac{1}{R}\int_0^1 {{P_{1ss}}(0,y)\left( {\rho (y,t) - {{\bar \rho }_1}} \right)dy} \\
 + v({{\bar \rho }_1})m_{1w}^Tw(t),{{\bar \rho }_1} \le \rho (1,t) \le {{\bar \rho }_2}\\
v({{\bar \rho }_2}){{\bar \rho }_2} - \frac{1}{R}\int_0^1 {{P_{2ss}}(0,y)\left( {\rho (y,t) - {{\bar \rho }_2}} \right)dy} \\
 + v({{\bar \rho }_2})m_{2w}^Tw(t),{{\bar \rho }_2} \le \rho (1,t) \le {{\bar \rho }_d}
\end{array} \right.
\end{equation}
with $m_{kw}^T = \frac{1}{{v({{\bar \rho }_k})R}}\int_0^1 {{P_{kss}}(0,y)m_k^T(y)dy}  + m_k^T(0)$, $k=1,2$.\\
While for the case: $d=2$, we need to solve the following equations for the control law in (\ref{new-cot}):\vspace{-2mm}
\[\begin{array}{*{20}{l}}
{v({{\bar \rho }_1}){\partial _z}{P_{1ss}}(z,y) + v({{\bar \rho }_1}){\partial _y}{P_{1ss}}(z,y)}\\
{ + {q_1}(z,y) - \frac{1}{R}{P_{1ss}}(z,0){P_{1ss}}(0,y) = 0}\\
{v({{\bar \rho }_1}){d_z}m_1^T(z) + m_1^T(z)S = 0}\\
{{P_{1ss}}(z,1) = 0,{P_{1ss}}(1,y) = 0}\\
{m_1^T(1) = \frac{1}{{v({{\bar \rho }_1})}}q_r^T}
\end{array}\]
and the control law is given by:
\begin{equation}\label{cou-2}
\begin{array}{l}
{u_2}(t) = v({{\bar \rho }_1}){{\bar \rho }_1} - \frac{1}{R}\int_0^1 {{P_{1ss}}(0,y)\left( {\rho (y,t) - {{\bar \rho }_1}} \right)dy} \\
 + v({{\bar \rho }_1})m_w^Tw(t)
\end{array}
\end{equation}
with $m_w^T = \frac{1}{{v({{\bar \rho }_1})R}}\int_0^1 {{P_{1ss}}(0,y)m_1^T(y)dy}  + m_1^T(0)$.
\begin{figure}[ht]
  \centering
  \includegraphics[width=3.5in]{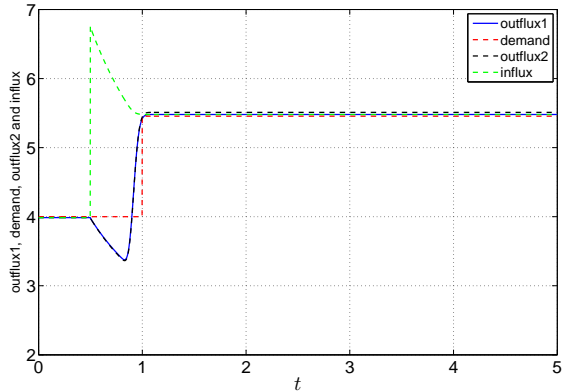}\vspace{-5mm}\\
  \caption{Influx (green dashed line), outflux(solid line and black dashed line), and demand rate (red dashed line) for a step demand function from 4 to 5.4545 at $t=1$.}\label{fig-4}
\end{figure}

From Figure \ref{fig-4}, in order to achieve the desired demand rate, we need to increase the influx (see green dashed line), which causes that the average velocity decreases. Therefore, the outflux will decrease first (see black dashed line and blue solid line) since the density $\rho(1,t)$ does not change but the velocity decreases. Note that in Figure \ref{fig-4}, the blue line denotes the outflux under the control law $u_1(t)$ in (\ref{cou-1}) and the black dashed line presents the outflux under the control law $u_2(t)$ in (\ref{cou-2}). It is shown that under the control $u_1(t)$, the outflux can reach to the desired demand rate more accurately.

\section*{Conclusion}
This work addressed the optimal demand rate tracking problem for the production manufacturing system based on a continuum model. Due to the complexity that the nonlinearity brings, all results of this work are derived based on the linearized production system model. To solve the optimal stabilization of the resulting linear production system model, a weak variation approach is utilized and the corresponding open-loop and state feedback boundary control laws are obtained. Furthermore, based on the resulting state feedback stabilization control law, the boundary tracking regulator is constructed. In particular, the optima stabilization yields optimal reference demand rate tracking and this conclusion can be obtained from the proof part of Theorem \ref{theorem-3}. Finally, the proposed approach is verified through a numerical example where we applied the resulting control law to the original nonlinear manufacturing system and the performance of the tracking controller is shown in Figure \ref{fig-4}.

\bibliographystyle{apalike}
\small
\bibliography{myreference}

\begin{thebibliography}{}

\bibitem[Armbruster et~al., 2006]{armbruster2006continuum}
Armbruster, D., Marthaler, D.~E., Ringhofer, C., Kempf, K., and Jo, T.-C.
  (2006).
\newblock A continuum model for a re-entrant factory.
\newblock {\em Operations research}, 54(5):933--950.

\bibitem[Athans and Falb, 1966]{athans1966optimal}
Athans, M. and Falb, P.~L. (1966).
\newblock {\em Optimal Control: An Introduction to the Theory and Its
  Applications}.
\newblock Courier Corporation.

\bibitem[Deutscher, 2015]{deutscher2015backstepping}
Deutscher, J. (2015).
\newblock A backstepping approach to the output regulation of boundary
  controlled parabolic {PDEs}.
\newblock {\em Automatica}, 57:56--64.

\bibitem[Dong et~al., 2011]{dong2011optimal}
Dong, M., He, F., and Wu, Z. (2011).
\newblock Optimal control of a continuum model for single-product re-entrant
  manufacturing systems.
\newblock {\em International Journal of Production Research},
  49(21):6363--6385.

\bibitem[Francis, 1977]{francis1977linear}
Francis, B.~A. (1977).
\newblock The linear multivariable regulator problem.
\newblock {\em SIAM Journal on Control and Optimization}, 15(3):486--505.

\bibitem[Xu and Dubljevic, 2016a]{xu2016finite}
Xu, X. and Dubljevic, S. (2016a).
\newblock Finite-dimensional output feedback regulator for a mono-tubular
  heatexchanger process.
\newblock {\em arXiv preprint arXiv:1605.04973}.

\bibitem[Xu and Dubljevic, 2016b]{xu2015output}
Xu, X. and Dubljevic, S. (2016b).
\newblock Output regulation problem for a class of regular hyperbolic systems.
\newblock {\em International Journal of Control}, 89(1):113--127.

\bibitem[Xu and Dubljevic, 2016c]{Xu2016servo}
Xu, X. and Dubljevic, S. (2016c).
\newblock The state feedback servo-regulator for countercurrent heat-exchanger
  system modelled by system of hyperbolic {PDEs}.
\newblock {\em European Journal of Control}, 29:51--61.

\bibitem[Xu et~al., 2015]{Xu2015finites}
Xu, X., Pohjolainen, S., and Dubljevic, S. (2015).
\newblock Finite-dimensional regulators for a class of regular hyperbolic {PDE}
  systems.
\newblock {\em Journal of Process Control, Under Review}.

\end{thebibliography}

\end{document}